\documentclass[12pt]{amsart}
\usepackage{amsmath}
\usepackage{amstext}
\usepackage{amssymb}
\usepackage{amsthm}
\usepackage{amsrefs}
\usepackage{imakeidx}
\usepackage{epsfig}
\usepackage{afterpage}
\usepackage{ifthen}
\usepackage{enumitem}
\usepackage{mathtools}
\mathtoolsset{centercolon}
\usepackage{subcaption}
\usepackage{array}
\usepackage{color}
\usepackage{hyperref}
\hypersetup{
    colorlinks=true,
    linkcolor=blue,
    filecolor=magenta,      
    urlcolor=cyan,
}
\usepackage{graphicx}            
\usepackage{microtype}
\usepackage{boldline}

\makeatletter
\def\@cite#1#2{{\m@th\upshape\bfseries%
[{#1\if@tempswa{\m@th\upshape\mdseries, #2}\fi}]}}
\makeatother
\theoremstyle{plain}
\newtheorem{thm}{Theorem}[section]
\newtheorem{cor}[thm]{Corollary}
\newtheorem{prop}[thm]{Proposition}
\newtheorem{lem}[thm]{Lemma}

\theoremstyle{definition}
\newtheorem{ex}[thm]{Example}
\newtheorem{rem}[thm]{Remark}

\numberwithin{equation}{section}
\newcommand{\bC}{{\mathbb{C}}}
\newcommand{\bD}{{\mathbb{D}}}

\newcommand{\bT}{{\mathbb{T}}}

\newcommand{\A}{{\mathcal{A}}}
\newcommand{\B}{{\mathcal{B}}}
\renewcommand{\H}{{\mathcal{H}}}
\newcommand{\M}{{\mathcal{M}}}
\newcommand{\N}{{\mathcal{N}}}

\newcommand{\T}{{\mathcal{T}}}
\newcommand{\ep}{\varepsilon}
\renewcommand{\phi}{\varphi}
\newcommand{\AND}{\text{ and }}

\newcommand{\qand}{\quad\text{and}\quad}
\newcommand{\qfor}{\quad\text{for}\ }
\newcommand{\qforal}{\quad\text{for all}\ }
\newcommand{\ol}{\overline}
\newenvironment{sbmatrix}{\left[\begin{smallmatrix}}{\end{smallmatrix}\right]}
\newcommand{\ip}[1]{\langle #1 \rangle}
\newcommand{\bip}[1]{\big\langle #1 \big\rangle}
\newcommand{\id}{{\operatorname{id}}}
\newcommand{\ind}{\operatorname{ind}}

\newcommand{\rank}{\operatorname{rank}}

\newcommand{\wot}{\textsc{wot}}

\begin{document}

%%%%%%%%%%%%%%%%%%%%%%%%%%%%%%%%
\title{Large Perturbations of Nest Algebras}

\author[K.R. Davidson]{Kenneth R. Davidson}
\address{Pure Mathematics Department\\
University of Waterloo\\
Waterloo, ON\; N2L--3G1\\
CANADA}
\email{krdavids@uwaterloo.ca}

\begin{abstract} 
Let $\M$ and $\N$ be nests on separable Hilbert space.
If the two nest algebras are distance less than 1 ($d(\T(\M),\T(\N)) < 1$), then the nests are distance less than 1 ($d(\M,\N)<1$).
If the nests are distance less than 1 apart, then the nest algebras are similar, 
i.e. there is an invertible $S$ such that $S\M = \N$, so that $S \T(\M)S^{-1} = \T(\N)$.
However there are examples of nests closer than 1 for which the nest algebras are distance 1 apart.
\end{abstract}



\thanks{Author partially supported by NSERC Grant RGPIN-2018-03973.}

\maketitle

%%%%%%%%%%%%%%%%%%%%%%%%%%%%%%%%
\section{Introduction}
%%%%%%%%%%%%%%%%%%%%%%%%%%%%%%%%

Kadison and Kastler \cite{KK} defined a distance between two operator algebras which is equivalent to (but slightly different than) 
the Hausdorff distance between their unit balls.
They showed that close factors have the same type, and began a long story on permanence of various properties under small perturbations.

In \cite{Lance}, Lance considered close nest algebras.
He showed that the nest algebras are close if and only if the nests are close.
Then using cohomological methods, he established that if two nests are sufficiently close,
then the corresponding nest algebras are similar via a similarity close to the identity.
Once the Similarity theorem for nests was established by the author \cite{Dav_sim}, it was possible to obtain easier proofs with better constants.
In \cite{Dav_pert}, the author showed that if a nest algebra is within $\gamma < .01$ of an arbitrary norm closed operator algebra,
then the other algebra is also a nest algebra and they are similar via an invertible operator $S$ satisfying $\|S-I\| < 8\gamma$.
These constants were improved in \cite{DavNestAlgs} to $\gamma < \frac1{20}$ and $\| S-I \| < 4\gamma$.
If both algebras are nest algebras, $\gamma < \frac15$ will suffice.
It is also shown there that if there is an order isomorphism  $\theta$ between two nests such that $\| \theta - \id \| = \gamma < \frac12$,
then the two nest algebras are similar via a similarity $S$ with $\| S-I \| \le 2 \gamma$, and $S$ implements $\theta$.

Usual results on perturbations examine the case where the perturbation is small.
Here we show that we can push the boundary to the largest possible, namely distance less than 1, since any two algebras/lattices are distance at most 1 apart.
We will show that if two nest algebras are at distance less than 1, then so are the corresponding nests.
And we show that if two nests are distance less than 1, then there is a dimension-preserving order isomorphism between them.
The Similarity theorem then shows the two nest algebras are similar.
However there are nests at distance $1/\sqrt2$ such that the nest algebras are distance 1.

Kadison-Kastler \cite{KK}  inspired a large literature on perturbations of algebras, especially between von Neumann algebras or C*-algebras. 
We mention work of Christensen \cites{Chr_pertI, Chr_pertII} and a counterexample of Choi and Christensen \cite{CC}.
There is more recent work of Christensen et al \cite{CSSWW} and Cameron et al \cite{CCSSWW}, two papers representative of their recent work.
The references to these recent papers mention other work in this direction.
Barry Johnson applied cohomological methods to perturbation problems for Banach algebras in \cite{Johnson}.
Perturbations of nonself-adjoint operator algebras were considered by Lance \cite{Lance}, the author \cite{Dav_pert}, Choi and the author \cite{CD},
and Pitts \cite{Pitts}, among others.

\medskip
In this paper, all Hilbert spaces will be separable.
We recall that a nest $\N$ is a set of subspaces of a Hilbert space $\H$ which is totally ordered by inclusion, contains $\{0\}$ and $\H$,
and is complete under arbitrary intersections and closed spans of subsets of the nest. 
We will frequently identify a nest $\N$ with the set of projections $\{ P_N : N \in \N \}$.
The nest algebra $\T(\N)$ is the \wot-closed unital algebra of all operators leaving each $N\in\N$ invariant.
For $N \in \N$, set $N_+ = \bigwedge \{ N' \in \N : N < N' \}$.
When $N_+ \ne N$, the projection $A = P_{N_+} - P_N$ is called an atom of $\N$.

A map $\theta:\M \to \N$ between nests is an order isomorphism if it is an order-preserving bijection.
Say that $\theta$ preserves dimension if 
\[
 \rank( P_{M_2} - P_{M_1} ) = \rank( P_{\theta(M_2)} - P_{\theta(M_1)} ) \qforal M_1 < M_2 \in \M.
\]
One can show that $\theta$ is order-preserving precisely when it preserves the rank of each atom of $\M$.
If $S \in \B(H)$ is invertible and $S\M = \N$, the map $\theta_S(M) = SM$ is a dimension-preserving order isomorphism.
The Similarity theorem \cite{Dav_sim} states that an order isomorphism which preserves dimension is implemented
by an invertible operator $S$. Moreover, if $\ep>0$, $S$ can be chosen of the form $S = U+K$ where $U$ is unitary, $K$ is compact and $\|K\| < \ep$.

These ideas can be found in my book \cites{DavNestAlgs}.
Some of the material in this note will appear in my forthcoming book \cite{DavFAOA}, as well as a simplified proof of the Similarity theorem.

The author thanks the referee for a careful reading.

%%%%%%%%%%%%%%%%%%%%%%%%%%%%%%%%
\section{Projections}
%%%%%%%%%%%%%%%%%%%%%%%%%%%%%%%%

The first task is to establish Theorem~\ref{T:two pairs proj}, which can be thought of as an index result.
The method is elementary, but care is required to obtain the sharp result.

%%%%%%%%%%%%%%%%%%%%%%%%%%
\begin{lem} \label {L:dist to pisom}
Suppose that $P$ and $Q$ are orthogonal projections such that $\|P-Q\| = \sin\theta < 1$ for $0 \le \theta < \frac\pi2$.
Let $U$ be the partial isometry in the polar decomposition of $QP$.
Then $\|U-P\| = 2 \sin\frac\theta2 < \sqrt2$.
\end{lem}

\begin{proof}
If $P \wedge Q = R \ne 0$, we can replace $P$ and $Q$ by $P-R$ and $Q-R$.
This does not affect the problem unless $P=Q$, which is a trivial case.
Let $P' = (P \vee Q) - P$. Then $P'QP$ is injective on $PH$ and has range dense in $P'H$.
There is no loss in assuming that $P+P'=I$.

Let $C = (PQP)^{1/2}$ and $S = (1-C^2)^{1/2}$.
The Halmos model \cite{Halmos} for two subspaces in generic position (i.e., $P\wedge Q = 0$ and $P \vee Q = I$) 
shows that $P$, $Q$ and $U$ are unitarily equivalent to
\[
 P = \begin{bmatrix}I&0\\0&0\end{bmatrix}
 \quad
 Q = \begin{bmatrix}C^2&CS\\CS&S^2\end{bmatrix} 
 \qand
 U = \begin{bmatrix}C&0\\ S&0\end{bmatrix} .
\]
Then since the two columns of $P-Q$ have orthogonal ranges,
\[
 \| P - Q \| = \Big\| \begin{bmatrix}S^2&-CS\\-CS&-S^2\end{bmatrix} \Big\| = \| S^2 ( S^2 +C^2) \|^{1/2} = \|S\| .
\]
So $\|S\| = \sin\theta$. Thus $C = (1-S^2)^{1/2} \ge \cos\theta I$ and $\| I - C \| = 1-\cos\theta$.
Therefore
\[
 \|U-P\| = \Big\| \begin{bmatrix}C-I&0 \\ S&0\end{bmatrix} \Big\| = \| (C-I)^2 + S^2 \|^{1/2} = \sqrt{\| 2(I-C)\|} .
\]
Hence $\|U-P\| = \sqrt{2(1-\cos\theta)} = 2 \sin \frac\theta2$.
\end{proof}

%%%%%%%%%%%%%%%%%%%%%%%%%%
\begin{thm} \label {T:two pairs proj}
Let $P_1$ and $P_2$ $($respectively $Q_1$ and $Q_2)$ be pairwise orthogonal projections on a separable Hilbert space.
Suppose that 
\[
 \| P_1 - Q_1 \|<1 \qand  \| P_2 - Q_2 \| < 1 . 
\]
Then
\[
 \rank (P_1+P_2)^\perp = \rank (Q_1+Q_2)^\perp .
\]
\end{thm}

\begin{proof}
We may suppose that $ \rank (P_1+P_2)^\perp \le \rank (Q_1+Q_2)^\perp$.
If both are infinite, there is nothing to prove.
Thus we may assume that $ \rank (P_1+P_2)^\perp$ is finite.
Form the polar decomposition 
\[
 Q_iP_i = U_i (P_iQ_iP_i)^{1/2} \qfor  i=1,2 .
\]
Then $U = U_1+U_2$ is a partial isometry of $(P_1+P_2)H$ onto $(Q_1+Q_2)H$.
Thus $U$ is left semi-Fredholm since it has finite nullity and closed range.
However by the lemma,
\begin{align*}
 \| U - (P_1+P_2) \|^2 &= \big\| \big[ (U_1-P_1)P_1\ \ \ (U_2-P_2)P_2  \big]  \big\|^2 \\&
 \le \|U_1-P_1\|^2  + \|U_2-P_2\|^2 < 4 
\end{align*}
because the two terms in the $1 \times 2 $ matrix have orthogonal domains.
If $\rank (Q_1+Q_2)^\perp = \infty$, then $\pi(U)$ would be a proper isometry, and thus $\sigma_e(U) = \ol{\bD}$ is the closed unit disc.
If $\rank (Q_1+Q_2)^\perp < \infty$, then $\pi(U)$ is unitary.
If also $\ind U \ne 0$, then $\sigma_e(U) = \ol{\bT}$ is the unit circle.
In either case,
\[
 \|U-(P_1+P_2)\|_e = \|\pi(U - I) \| = 2 .
\]
Hence $\ind U = 0$, which means  $\rank (P_1+P_2)^\perp = \rank (Q_1+Q_2)^\perp$.
\end{proof}

%%%%%%%%%%%%%%%%%%%%%%%%%%
\begin{cor} \label {C:closest proj}
If $P$, $Q_1$ and $Q_2$ are projections with $Q_1 < Q_2$, then 
\[
\max\{ \|P - Q_i \| : i=1,2 \} = 1 .
\]
\end{cor}

\begin{proof}
If this result were false, we could have $\|P - Q_i \| < 1$ for $i=1,2$.
Using $P$, $P^\perp$, $Q_1$ and $Q_2^\perp$ in Theorem~\ref{T:two pairs proj}, we would obtain
\[
 0 = \rank(P+P^\perp)^\perp = \rank(Q_1+Q_2^\perp)^\perp = \rank (Q_2 - Q_1) > 0.
\]
This contradiction establishes the result.
\end{proof}

Now we examine the difference between two projections.

%%%%%%%%%%%%%%%%%%%%%%%%%%
\begin{prop} \label {P:dist betw projs}
Let $P$ and $Q$ be projections on a Hilbert space $\H$.
Then 
\[
 \| P - Q \| = \max \big\{ \|PQ^\perp\|, \| P^\perp Q \| \big\} .
\]
Moreover if $\|P-Q\|<1$, then $\|PQ^\perp\| = \| P^\perp Q \|$.
\end{prop}

\begin{proof}
First of all, $P-Q = PQ^\perp - P^\perp Q$ is the difference of two operators with pairwise orthogonal domains and ranges.
Thus 
\[
 \| P - Q \| = \max \big\{ \|PQ^\perp\|, \| P^\perp Q \| \big\} .
\]

We make use of the general form for a pair of projections \cite{Halmos}. 
There is a canonical decomposition $\H = \H_{00} \oplus \H_{10} \oplus \H_{01} \oplus \H_{11} \oplus \H_2 \oplus \H_2$
and positive injective operators $C,S \in \B(H_2)$ such that $C^2+S^2 = I$ and
\[
 P = 0 \oplus I \oplus 0 \oplus I \oplus \begin{bmatrix} I & 0 \\0 & 0\end{bmatrix} 
 \ \AND \ 
 Q = 0 \oplus 0 \oplus I \oplus I \oplus \begin{bmatrix} C^2 & CS \\CS & S^2\end{bmatrix} .
\]
Any of the spaces $\H_{ij}$ and $\H_2$ can be vacuous.
It follows that 
\[
 P Q^\perp = 0 \oplus I \oplus 0 \oplus 0 \oplus \begin{bmatrix} S^2 & -CS \\0 & 0\end{bmatrix} 
 \ \AND \ 
 P^\perp Q = 0 \oplus 0 \oplus I \oplus 0 \oplus \begin{bmatrix} 0 & 0 \\CS & S^2\end{bmatrix} .
\]
Now 
\[
 \left\| \begin{bmatrix} S^2 & -CS \\0 & 0\end{bmatrix} \right\| =
 \left\| \begin{bmatrix} 0 & 0 \\CS & S^2\end{bmatrix} \right\| =
 \| S^2 ( S^2+C^2) \|^{1/2} = \|S\|. 
\]

If $\| P - Q \| < 1$, then $\H_{10}$ and $\H_{01}$ must be vacuous.
Hence 
\[
 \|PQ^\perp\| = \| P^\perp Q \| = \|S\| .
\]
However if either $\H_{10}$ or $\H_{01}$ is non-trivial, then $\|PQ^\perp\| = 1$ or $\|PQ^\perp\| = 1$, respectively.
If say $\H_{10} \ne \{0\}$ and $\H_{01} = \{0\}$, then $\|PQ^\perp\|=1$ and $\|P^\perp Q\| = \|S\|$.
So if $\|S\|<1$, the two norms differ. It is similar if $\H_{10} = \{0\}$ and $\H_{01} \ne \{0\}$.
\end{proof}

%%%%%%%%%%%%%%%%%%%%%%%%%%%%%%%%
\section{Distance between nests}
%%%%%%%%%%%%%%%%%%%%%%%%%%%%%%%%

We define the distance between two nests to be the Hausdorff distance between the sets of corresponding orthogonal projections.
If $\M$ and $\N$ are nests, set
\[
 d(\M, \N) = \max \Big\{ \sup_{N \in \N} \inf_{M\in \M} \| P_N - P_M \|,\ \sup_{M\in \M} \inf_{N \in \N} \| P_N - P_M \| \Big\} .
\]

%%%%%%%%%%%%%%%%%%%%%%%%%%
\begin{thm} \label {T:close nests}
Suppose that $\M$ and $\N$ are two nests on $\H$ such that $d(\M, \N) = \gamma < 1$.
Then there is a unique order isomorphism $\theta : \M \to \N$ such that $\| \theta - \id \| = \gamma$.
Moreover $\theta$ preserves dimension.
\end{thm}

\begin{proof}
For each $M\in \M$, there is an $N \in \N$ with $\| P_M - P_N \| \le \gamma$.
By Corollary~\ref{C:closest proj}, the subspace $N$ is unique.
Define $\theta(M) = N$.
Reversing the roles of $\M$ and $\N$ shows that $\theta$ is a bijection.
It is evident that this map must preserve order, so $\theta$ is an order isomorphism.

Suppose that $M_1 < M_2$, $P_i := P_{M_i}$ and $Q_i = P_{\theta(M_i)}$ for $i=1,2$.
Then $\|P_1 - Q_1\| < 1$, $\| P_2^\perp - Q_2^\perp\| = \|P_2 - Q_2 \| < 1$, $P_1$ and $P_2^\perp$ are orthogonal 
and $Q_1 < Q_2$, so that $Q_1$ and $Q_2^\perp$ are orthogonal.
Hence by Theorem~\ref{T:two pairs proj}, 
\begin{align*}
 \rank(P_2-P_1) &= \rank (P_1+P_2^\perp)^\perp \\& = \rank (Q_1+Q_2^\perp)^\perp = \rank(Q_2-Q_1) . 
\end{align*}
This means that $\theta$ preserves dimension.
\end{proof}

The following result was established by Lance when $\gamma$ is sufficiently small.
In \cite{DavNestAlgs}*{Lemma 18.3}, this was extended to $\gamma < \frac12$.
Here we see that $\gamma<1$ is the sharp dividing point.

%%%%%%%%%%%%%%%%%%%%%%%%%%
\begin{cor}
Suppose that $\theta : \M \to \N$ is an order isomorphism of two nests on separable Hilbert space.
If 
\[
 \|\theta-\id\| := \sup_{M\in\M} \| P_{\theta(M)} - P_M \| = \gamma < 1 ,
\]
then there is an invertible operator $S$ 
such that $SM = \theta(M)$ for all $M \in \M$.
If $\gamma < \frac12$, one can take $\|S-I\| < 2\gamma$.
If $\gamma \ge \frac12$, one can take $\|S-I\| < 2 + \ep$ for any $\ep>0$.
\end{cor}

\begin{proof}
By Theorem~\ref{T:close nests}, $\theta$ preserves dimension.
Hence by the Similarity theorem \cite{Dav_sim}, $\theta$ is implemented by a similarity. 
(See the introduction for some discussion of this.)
When $\gamma < \frac12$, \cite{DavNestAlgs}*{Lemma 18.3} shows that one can choose $S$ so that $\| S - I \| < 2 \gamma$.
For $\gamma \ge \frac12$, one can take $S = U+K$ with $U$ unitary and $\|K\| < \ep < \frac12$.
Hence 
\[
 \|S-I \| \le \| U \| + \| K \| + \| I \|  < 2 + \ep < 5 \gamma . \qedhere
\]
\end{proof}

%%%%%%%%%%%%%%%%%%%%%%%%%%%%%%%%
\section{Distance between nest algebras}
%%%%%%%%%%%%%%%%%%%%%%%%%%%%%%%%

Kadison and Kastler \cite{KK} defined the distance between two operator algebras $\A$ and $\B$ as
\[
 d(\A, \B) = \max \Big\{ \sup_{A \in b_1(\A)} \inf_{B \in \B} \|A-B\|, \sup_{B \in b_1(\B)} \inf_{A \in \A} \|A-B\| \Big\} .
\]
This is slightly different than the Hausdorff distance between the two unit balls, 
but they are comparable within a factor of 2.
We have a mild preference for this metric due to the Arveson distance formula \cite{Arv}: 
\[
 d(T, \T(\N) ) = \sup_{N\in\N} \| P_N^\perp T P_N \| \qfor T\in\B(\H) .
\]
See Power \cite{Power} and Lance \cite{Lance} for other proofs.
In any case, 
\[
 0 \le d(\A,\B) \le d_H(b_1(\A), b_1(\B)) \le 1 .
\]
So if $d(\A,\B) = 1$, then $d_H(\A,\B) = 1$ as well.

Lance \cite{Lance} showed that if two nest algebras are sufficiently close, then the nests are close and vice versa.
Using the Similarity theorem, the author quantified this in \cite{Dav_pert}.
The best known constants \cite{DavFAOA}*{Theorem 13.10.11} are

%%%%%%%%%%%%%%%%%%%%%%%%%%%%%%%%
\begin{thm}[Lance, Davidson] \label {T:lance}
Let $\M,\N$ be nests on separable Hilb\-ert space $\H$.
If $d(\T(\M), \T(\N)) = \gamma < \frac15$, then $\T(\M)$ and $\T(\N)$ are similar via an invertible operator $S$
satisfying $\| S - I \| < 4\gamma$.
\end{thm}

To get a handle on the distance between nest algebras, we make use of the fact that for any $N\in\N$,
$\T(\N)$ contains $P_{N_+} \B(H) P_N^\perp$. In particular, all rank one elements of $\T(\N)$ have this form.
We will write $\zeta \eta^*$ for the rank 1 operator $(\zeta \eta^*) \xi = \ip{\xi,\eta} \zeta$.

%%%%%%%%%%%%%%%%%%%%%%%%%%%%%%%%
\begin{lem} \label {L:lower bound}
Let $\M, \N$ be nests on $\H$. For any $M\in\M$ and $N\in\N$,
\[
 d(\T(\M), \T(\N)) \ge \| P_M^\perp P_{N_+} \| \, \|P_M P_N^\perp \| .
\] 
\end{lem}

\begin{proof}
Let $\ep > 0$.
We may choose a unit vector $\eta \in N^\perp$ so that $\|P_M \eta \| > \|P_M P_N^\perp \| - \ep$.
Likewise choose a unit vector $\zeta \in N_+$ so that $\| P_M^\perp \zeta \| > \| P_M^\perp P_{N_+} \| - \ep$.
The rank one operator $\zeta \eta^*$ lies in $\T(\N)$ and has norm 1.
Therefore
\begin{align*}
 d(\T(\M), \T(\N)) &\ge d( \zeta \eta^*, \T(\M)) \ge \| P_M^\perp \zeta \eta^* P_M \| \\&
 = \| P_M^\perp \zeta \| \,  \|P_M \eta \| > \| P_M^\perp P_{N_+} \| \, \|P_M P_N^\perp \| - 2\ep. 
\end{align*}
As $\ep>0$ is arbitrary, the result follows.
\end{proof}

As $P_N \le P_{N_+}$, the following is immediate.

%%%%%%%%%%%%%%%%%%%%%%%%%%%%%%%%
\begin{cor} \label {C:lower bound}
Let $\M, \N$ be nests on $\H$. For any $M\in\M$ and $N\in\N$,
\[
 d(\T(\M), \T(\N)) \ge \| P_M^\perp P_N \| \, \|P_M P_N^\perp \| .
\] 
\end{cor}

We can now make the following improvement to Theorem~\ref{T:lance}, pushing $\gamma$ to the optimal $\gamma < 1$.

%%%%%%%%%%%%%%%%%%%%%%%%%%%%%%%%
\begin{thm} \label {T:dist algs}
Let $\M,\N$ be nests on separable Hilbert space $\H$.
If 
\[
 d(\T(\M), \T(\N)) < 1 ,
\]
then $d(\M,\N) < 1$; and so $\T(\M)$ and $\T(\N)$ are similar.
\end{thm}

\begin{proof}
For the contrapositive, we suppose that $d(\M,\N) = 1$ and show that $d(\T(\M), \T(\N)) = 1$.

There are two cases to consider. In the first case, there are projections $M_i\in\M$ and $\N_i \in \N$
such that $\| P_{M_i} - P_{N_i} \| = t_i < 1$ for $i\ge1$ with $\sup t_i = 1$.
In this case, Proposition~\ref{P:dist betw projs} shows that $\| P_{M_i}^\perp P_{N_i} \| = \| P_{M_i} P_{N_i}^\perp \| = t_i$.
Thus by Corollary~\ref{C:lower bound},
\[
 1 \ge d(\T(\M), \T(\N)) \ge \sup_{i\ge 1} \| P_{M_i}^\perp P_{N_i} \| \, \| P_{M_i} P_{N_i}^\perp \| = \sup_{i\ge1} t_i^2 = 1.
\]

In the second case, there is a projection in one nest at distance 1 from the other.
Interchanging the two nests if required, we may suppose that there is $M \in \M$ such that $\| P_M - P_N \| = 1$ for all $N\in \N$.
Let 
\[
 \delta = \sup_{N\in \N} \min \big\{ \| P_M^\perp P_N \| , \|P_M P_N^\perp \| \big\} .
\]
Again Corollary~\ref{C:lower bound} and Proposition~\ref{P:dist betw projs} show that $d(\T(\M), \T(\N)) \ge \delta$.
Hence if $\delta = 1$, we are done.

So we assume that $\delta < 1$.
If $\| P_M^\perp P_N \| \le \delta$, the same holds for all $N'<N$ in $\N$.
Define $N_0 = \bigvee \{ N : \| P_M^\perp P_N \| \le \delta \}$.
There are $N_i$ increasing to $N_0$ with $\| P_M^\perp P_{N_i} \| \le \delta$.
Since $P_{N_i}$ converges in the strong operator topology to $P_{N_0}$, we have $\| P_M^\perp P_{N_0} \| \le \delta$.
That is, $N_0$ is the largest element of $\N$ with this property. Hence $\|P_M P_{N_0}^\perp \| = 1$.

If $N>N_0$, then $\| P_M^\perp P_N \| = 1$, and thus $ \|P_M P_N^\perp \| \le \delta$.
Recall that $N_{0+} = \bigwedge \{ N \in \N : N > N_0 \}$.
Arguing as above, we obtain $ \|P_M P_{N_{0+}}^\perp \| \le \delta$.
Thus $\| P_M^\perp P_{N_{0+}} \| = 1$. This shows that $P_{N_{0+}} > P_{N_0}$.
Now an application of Lemma~\ref{L:lower bound} shows that 
\[
 1 \ge d(\T(\M), \T(\N)) \ge \| P_M^\perp P_{N_{0+}} \| \, \|P_M P_{N_0}^\perp \| = 1 . \qedhere
\]
\end{proof}

Lance \cite{Lance} showed that if $d(\M,\N) = \gamma < \frac12$, then $d(\T(\M), \T(\N)) \le 2 \gamma < 1$.
We now provide an elementary example to show that it is possible to have $\gamma < 1$ but $d(\T(\M), \T(\N)) = 1$.

\begin{ex}
Let $\frac1{\sqrt2} \le s < 1$ and set $c = \sqrt{1-s^2}$. Let $\H = \bC^2$.  
Let $M= \bC \begin{sbmatrix} 1\\0 \end{sbmatrix}$ and $N = \bC \begin{sbmatrix} c\\s \end{sbmatrix}$
and define nests $\M = \big\{ \{0\}, M , \H \big\}$ and $\N = \big\{ \{0\}, N , \H \big\}$.
Then $\T(\M)$ is the algebra of $2\times2$ upper triangular matrices in the standard basis.
Using the calculation in Lemma~\ref {L:dist to pisom}, 
\[
 d(\M,\N) = \| P_M - P_N \| = s < 1 .
\]

For $0 \le a \le 1$, consider $T = \begin{bmatrix} a & 1-a^2 \\ 0 & -a \end{bmatrix}$ in $\T(\M)$.
One sees that
\[
 \|T\| = \left\| \begin{bmatrix} a & 1-a^2 \\ 0 & a \end{bmatrix} \right\|
 = \left\| \begin{bmatrix} 0 & a \\ a & 1-a^2 \end{bmatrix} \right\| = 1
\]
because the last matrix, which is self-adjoint, has characteristic polynomial $\lambda^2 - (1-a^2) \lambda - a^2 = (\lambda - 1)(\lambda + a^2)$.
We compute 
\begin{align*}
 d(T, \T(\N)) &= \| P_N^\perp T P_N \| = \big| \bip{ T \begin{sbmatrix} c\\s \end{sbmatrix}, \begin{sbmatrix} s\\ -c \end{sbmatrix} } \big| %\\&
 = 2acs + (1-a^2) s^2 .
\end{align*}
A bit of calculus shows that we should choose $a = \frac c s$, which yields
\[
 d(T, \T(\N)) = 2c^2 + s^2 - c^2 = c^2 + s^2 = 1.
\]
Therefore $d(\T(\M), \T(\N)) = 1$.
\end{ex}

%%%%%%%%%%%%%%%%%%%%%%%%%%%%%%%%
\begin{rem}
In the example above, when $0 < s < \frac1{\sqrt2}$, we can choose $a=1$ to see that $d(\T(\M), \T(\N)) \ge 2cs$.
For very small $s$, this shows that the constant 2 in Lance's estimate cannot be improved.
\end{rem}

%%%%%%%%%%%%%%%%%%%%%%%%%%%%%%%%

\end{document}